\documentclass[requno,12pt]{amsart}

\usepackage{amssymb}

\usepackage{enumitem}

\usepackage{graphicx}
\usepackage{mathrsfs}
\usepackage{pifont}
\usepackage{pgfplots}
\usepgfplotslibrary{fillbetween}
\makeatletter
\@namedef{subjclassname@2020}{%
  \textup{2020} Mathematics Subject Classification}
\makeatother

\usepackage[width=16.00cm, height=21.00cm, left=2.00cm, right=2.00cm, top=4cm, bottom=2cm]{geometry}

\usepackage[T1]{fontenc}

\newtheorem{theorem}{Theorem}[section]
\newtheorem{corollary}[theorem]{Corollary}

\newtheorem{proposition}[theorem]{Proposition}

\theoremstyle{definition}
\newtheorem{definition}[theorem]{Definition}
\newtheorem{remark}[theorem]{Remark}
\newtheorem{example}[theorem]{Example}

\numberwithin{equation}{section}

\def\cb{{\mathbb{C}}}

\def\cb{{\mathbb{C}}}

\newcommand{\ds}{\displaystyle}

\def\bR{{\Bbb R}}
\def\1{{\mathchoice {\rm 1\mskip-4mu l} {\rm 1\mskip-4mu l}{\rm 1\mskip-4.5mu l} {\rm 1\mskip-5mu l}}}
\DeclareMathOperator{\Sh}{Sh}
\DeclareMathOperator{\loc}{loc}

\title[Slices and $m$-Lelong numbers of  $m$-subharmonic functions]{Slices and  $m$-Lelong numbers  of  $m$-subharmonic functions}

\author[H. Khedhiri and N. Ghiloufi]{ Hedi Khedhiri\;$^{1}$ and Noureddine Ghiloufi\;$^{2}$}
\address{$^{1}$ 
  University of Monastir\\
 Preparatory Institute for Engineering Studies, LR18ES16 Research Laboratory, Monastir Tunisia }
  \address{$^{2}$ University of Gabes\\ Faculty of Sciences of Gabes\\ LR17ES11 Mathematics and Applications laboratory\\ 6072, Gabes, Tunisia.}
\email{$^{1}$ khediri\_h@yahoo.fr , hedi.khedhiri.math@gmail.com }
\email{$^{2}$ noureddine.ghiloufi@fsg.rnu.tn}


\begin{document}
\begin{abstract}
We investigate slicing properties of $m$-subharmonic functions in product domains $\Omega = \Omega' \times \Omega'' \subset \mathbb{C}^n = \mathbb{C}^p \times \mathbb{C}^{n-p}$, where $p, m, n$ are integers satisfying $1 \leq p \leq m-1 < n-1$.\\
Given an $m$-subharmonic function $v$ on $\Omega$, we prove the existence of a pluripolar subset $E \subset \Omega'$ such that, for every $x' \in \Omega' \smallsetminus E$, the slice $v_{|\{x'\}\times \mathbb{C}^{n-p}}$ is well defined and $(m - q_{m,p})$-subharmonic on $\Omega''$, where $q_{m,p}$ denotes the smallest integer greater than or equal to $\frac{mp}{n}$.\\
Moreover, we show that, outside a negligible subset of $\Omega'$, the $m$-Lelong number of $v$ at $(x', x'')$ coincides, up to a multiplicative constant, with the $(m - q_{m,p})$-Lelong number of the slice $v_{|\{x'\}\times \Omega''}$ at $x''$.
\end{abstract}

\maketitle

\section{Introduction and overview of the results}
The motivation for the present work stems in part from the relationship between the B{\l}ocki conjecture (see \cite{Bl}) on the integrability index of an $m$-subharmonic function at a point $a$ and its $m$-Lelong number at that point. More precisely, it is observed in \cite{Ben-Gh} that the resolution of the conjecture can be reduced to the case of points where the $m$-Lelong number is nonzero, and in such a case, the conjecture is proved for $m=1$.\\
Our aim is therefore to reduce the general case to lower values of $m$. To this end, we employ slicing techniques in our approach.\\

In this paper, we focus on the existence and structural properties of slices of $m$-subharmonic functions, as well as on the behavior of their associated $m$-Lelong numbers along lower-dimensional complex subspaces.\\

In the plurisubharmonic case $(m=n)$, slicing theory is classical, with well-developed connections to Lelong numbers and complex Monge--Amp\`ere measures.
In contrast, for general $m$-subharmonic functions $(1\leq m<n)$, slicing theory is far less developed, mainly because
$m$-subharmonicity is not preserved under arbitrary holomorphic coordinate changes.\\

Our results extend and complement several earlier works \cite{Bl, Ben-Gh, BMS-El, HK1, HK2, HK4, HK8, HK2bbis}, and contribute to a deeper understanding of $m$-subharmonic functions and their singularities. \\
In \cite{Bl}, the focus is on weak solutions to the complex Hessian equation and on integrability indices, without addressing slicing phenomena. 
The results of \cite{Ben-Gh} concern $m$-subharmonic functions in specific geometric settings, but do not involve lower-dimensional slices. 
In \cite{BMS-El}, slicing is studied in the plurisubharmonic case, whereas our analysis applies to the more general range $1 \leq p < m < n$.\\ 
In \cite{HK2, HK8}, the author establishes slicing properties for certain negative plurisubharmonic currents with small support, particularly those associated with analytic subsets. He studies the integrability of their coefficients and of their slices, but without relating the slice structure to Lelong numbers.\\

The results of this paper provide a unified treatment of the existence of slices, pluripolar exceptional sets, monotonicity formulas, and the relationship between the $m$-Lelong numbers of an $m$-subharmonic function $v$ and those of its slices.\\

Let $n$, $m$, and $p$ be integers satisfying $1 \leq p < m < n$, and let
$$
\Omega = \Omega' \times \Omega'' \subset \mathbb{C}^n = \mathbb{C}^p \times \mathbb{C}^{n-p}, \quad z = (z', z''), \quad z' \in \Omega', \quad z'' \in \Omega'',
$$
be a bounded domain in $\mathbb{C}^n$. We denote by $\pi : \Omega \to \Omega'$ the first canonical projection  defined by $\pi(z', z'') = z'$.\\

Theorem~\ref{t3.3} is our first main result. It is a slicing theorem for $m$-subharmonic functions that establishes the existence and structure of their slices. More precisely, we show that for any $m$-subharmonic function $v$, there exists a pluripolar subset $E \subset \Omega'$ such that, for every $x' \in \Omega' \smallsetminus E$, the slice $\langle v, \pi, x' \rangle$ is well defined and coincides with the natural restriction $v_{|\{x'\} \times \Omega''}$. Moreover, the associated currents
$$
\langle dd^c v, \pi, x' \rangle = dd^c \bigl( v_{|\{x'\} \times \Omega''} \bigr)
$$
inherit $(m - q)$-subharmonicity properties for every integer $q \geq \frac{mp}{n}$.\\

In Proposition~\ref{p3.1}, we introduce a directional $(m - q)$-Lelong function $\nu_{m-q}(dd^c v, \mathcal B, x'', \cdot)$ associated with the $m$-positive current $dd^c v$ at a point $x'' \in \Omega''$, where $\mathcal B$ is a bounded Borel subset of $\Omega'$. We prove that this function is non-decreasing, which ensures the existence of the directional $(m - q)$-Lelong number $\nu_{m-q}(dd^c v, \mathcal B, x'')$.\\
The proof of this proposition relies on several techniques that are of independent interest and yields new estimates for the Monge--Amp\`ere operator. These estimates are then used to establish the second main result of the paper.\\
More precisely, Theorem~\ref{t4.4} shows that, up to a positive multiplicative constant, the $m$-Lelong number of $v$ coincides with the $(m - q_{m,p})$-Lelong number of the slice $v_{|\{x'\} \times \Omega''}$ in the generic case where $q_{m,p} = \frac{mp}{n}$. When $\frac{mp}{n}$ is not an integer, and consequently $q_{m,p} > \frac{mp}{n}$, we prove that the $(m - q_{m,p})$-Lelong number of the slice actually vanishes. \\

The paper is organized as follows. Section~1 provides an introduction and presents the main results. Section~2 contains the necessary preliminaries. In Section~3, we establish several slicing properties of $m$-subharmonic functions and prove the first main result (Theorem~\ref{t3.3}). Finally, Section~4 is devoted to the study of $m$-Lelong numbers on slices: we prove their existence, derive their main properties, and establish sharp estimates leading to the second main result (Theorem~\ref{t4.4}).

\section{Preliminaries}
We use the standard notation for the differential operators
\begin{center}
$d = \partial + \bar{\partial} \quad \text{and} \quad d^c = i(\bar{\partial} - \partial), \quad \text{so that} \quad dd^c = 2i\,\partial \bar{\partial}.$
\end{center}

The K\"ahler form on $\mathbb{C}^n$, denoted by $\beta := \beta(t) = dd^c |t|^2$, decomposes as
$$
\beta(t) = \beta'(t') + \beta''(t''),
$$
where $\beta'$ and $\beta''$ denote the standard K\"ahler forms on $\mathbb{C}^p$ and $\mathbb{C}^{n-p}$, respectively.\\

For every $r > 0$ and $a \in \mathbb{C}^\ell$, we denote by $\mathbb{B}_\ell(a,r)$ the open ball in $\mathbb{C}^\ell$ centered at $a$ with radius $r$. In the particular case $a = 0$, we omit the center and write $\mathbb{B}_\ell(r)$ instead of $\mathbb{B}_\ell(0,r)$. \\

 The space $\mathscr D'_{k,k}(\Omega)$ of $(n-k,n-k)$-currents (i.e., currents of bidimension $(k,k)$) on $\Omega$,  defined as the dual of the space   $\mathscr D_{k,k}(\Omega)$ which consists of smooth, compactly supported $(k,k)$-forms  on $\Omega$.

  Following \cite{Wan, Ben-Gh}, we recall the $m$-positivity as follows:
  \begin{definition}
      \begin{enumerate}
      \item A  $(1,1)$-form $\alpha$ on $\Omega$ is said to be $m\text{-}$positive,  if  the wedge product $$\alpha^j\wedge\beta^{n-j}$$ defines a positive current for every $1\leq j\leq m$.
  \item  A $(q, q)$-form $\alpha$ on $\Omega$ is said to be strongly $m$-positive if $\alpha$ can be decomposed as $$\alpha=\sum_{j=1}^Na_j\alpha_{1,j}\wedge\dots\wedge\alpha_{q,j}$$
where $N=\binom{n}{q}$ and $\alpha_{1,j},\dots \alpha_{q,j}$ are $m$-positive $(1, 1)$-forms and $a_j\geq 0$ for every $j$.
\item  A current $T$ of bidimension $(k, k)$ on $\Omega$ with $m+k \geq n$ is said to be $m$-positive if $$\langle T\wedge \beta^{n-m}, \alpha \rangle\geq 0$$ for every strongly $m$-positive $(m +k-n, m +k-n)$-test form $\alpha$ on $\Omega$.
\item A current $T$ of bidimension $(k, k)$ on $\Omega$ with $m +k-1 \geq n$ is said to be $m$-subharmonic ($m\text{-}\Sh$ for short) if $dd^cT$ is an $m$-positive current.  
  \end{enumerate}
  \end{definition}
  
\begin{definition}
    A function $v:\Omega\longrightarrow[-\infty,+\infty[$ is said to be subharmonic on the domain $\Omega$ if $v\not\equiv-\infty$ and satisfies:
    \begin{enumerate}
        \item $v$ is upper semi-continuous on $\Omega$ and 
        \item for every $a\in\Omega$ and $0<r<\mathrm{dist}(a,\partial \Omega)$, 
        $$v(a)\leq \mathcal M(v,\mathbb S(a,r))=\frac{(n-1)!}{2\pi^nr^{2n-1}}\int_{\mathbb S(a,r)}v(z)d\sigma_n(z)$$
        or equivalently $$v(a)\leq\mathcal M(v,\mathbb B(a,r))=\frac{n!}{\pi^nr^{2n}}\int_{\mathbb B(a,r)}v(z)dV_n(z)$$ where  $\mathcal M(v,\mathbb S(a,r))$   and  $\mathcal M(v,\mathbb B(a,r))$  are the mean values of $v$ over the sphere and the ball, respectively.
    \end{enumerate}
    We set $Sh(\Omega)$ to be the set of subharmonic functions on $\Omega$.
\end{definition}
To describe the cone of $m$-subharmonic functions, we first recall some algebraic notions.\\

For any $\lambda = (\lambda_1,\dots,\lambda_n) \in \mathbb{R}^n$, we denote by $S_k(\lambda)$ the $k$-th elementary symmetric polynomial, defined by
$$S_k(\lambda) := \sum_{1 \leq j_1 < \cdots < j_k \leq n} \lambda_{j_1}\cdots \lambda_{j_k}.$$
For each $1 \leq m \leq n$, the G{\aa}rding cone $\Gamma_m$  (see \cite{Ga, S-A}) is defined by
$$\Gamma_m := \{ \lambda \in \mathbb{R}^n \; ; \; S_k(\lambda) \geq 0 \text{ for all } k = 1,\dots,m \}.$$
It is well known that $\Gamma_m$ is a convex cone in $\mathbb{R}^n$ and that the following inclusions hold:
$$\Gamma_n \subset \Gamma_{n-1} \subset \cdots \subset \Gamma_1.$$

As an example, let $v$ be a $\mathcal{C}^2$ subharmonic function on $\Omega$. Then its complex Hessian matrix
$$A(v) = \left(\frac{\partial^2 v}{\partial z_j \partial \overline{z}_\ell}\right)_{1 \leq j,\ell \leq n}$$
is Hermitian with nonnegative trace. In particular, its eigenvalue vector $\lambda(v) = (\lambda_1,\dots,\lambda_n)$ belongs to $\Gamma_1$.\\
Moreover, one has the identity (see \cite{S-A})
$$\binom{n}{k}(dd^c v)^k \wedge \beta^{n-k} = S_k(\lambda(v))\, \beta^n.$$
This shows that the positivity of the currents $(dd^c v)^k \wedge \beta^{n-k}$ is equivalent to the condition $\lambda(v) \in \Gamma_k$.\\
Consequently, the class of subharmonic functions can be decomposed into subclasses $\mathrm{Sh}_m(\Omega)$, where $v \in \mathrm{Sh}_m(\Omega)$ means that $v$ is subharmonic on $\Omega$ and $$(dd^c v)^k \wedge \beta^{n-k} \geq 0\qquad \text{for all}\quad 1 \leq k \leq m$$ in the sense of currents.
In particular, if $\mathrm{Psh}(\Omega)$ denotes the set of plurisubharmonic functions on $\Omega$, then
$$\mathrm{Psh}(\Omega) = \mathrm{Sh}_n(\Omega) \subsetneq \cdots \subsetneq \mathrm{Sh}_1(\Omega) = \mathrm{Sh}(\Omega).$$
\begin{remark}
\begin{enumerate}
    \item If $v$ is an $m$-subharmonic function on $\Omega$, then it is locally integrable and  $dd^cv$ is an $m$-positive current on $\Omega$.
    \item The fundamental solution of the complex Hessian equation
$$(dd^c u)^m \wedge \beta^{n-m} = \delta_0$$
is given by the $m$-subharmonic function $\Phi_{n,m}$ defined on $\mathbb{C}^n$ by $\Phi_{n,m}(z)=\phi_{n,m}(|z|^2)$ where 
$$\phi_{n,m}(r) := -\frac{1}{\left(\frac{n}{m}-1\right)r^{\frac{n}{m}-1}}.$$
This function plays a central role in our analysis.
\end{enumerate}    
\end{remark}

$m$-subharmonic functions have been extensively studied and are now well understood. In contrast, the notion of $m$-positive currents is less developed and requires deeper tools from complex analysis and geometry.\\

One of the most fundamental notions in complex analysis and geometry is that of Lelong numbers, originally introduced by Lelong in \cite{Le} for positive currents (the case $m = n$), and later significantly developed by several authors, notably Demailly, who introduced the generalized Lelong numbers (now commonly referred to as Lelong--Demailly numbers). These invariants measure the singularities of plurisubharmonic functions and analytic sets.\\

For $1 \leq m < n$, the authors of \cite{Ben-Gh} study $m$-Lelong numbers associated with $m$-positive currents. Some of their results will be used in this paper; therefore, we briefly recall the main ideas.\\

Let $T$ be an $m$-positive current of bidimension $(k,k)$ on $\Omega$, and let $a \in \Omega$. For $0 < r < \mathrm{dist}(a,\partial\Omega)$, the $m$-Lelong function of $T$ at $a$ is defined by
$$
\nu_m(T,a;r) := \frac{1}{r^{\frac{2n}{m}(m+k-n)}} \int_{\mathbb{B}_n(a,r)} T \wedge \beta^k.
$$
The $m$-Lelong number of $T$ at $a$, when it exists, is given by
$$
\nu_m(T,a) := \lim_{r \to 0} \nu_m(T,a;r).
$$
It is shown in \cite{Ben-Gh} that if $T$ is an $m$-positive $m$-subharmonic current on $\Omega$, then the $m$-Lelong number of $T$ exists at every point $a \in \Omega$. 

The special case where $T = dd^c v$ for some $m$-subharmonic function $v$ on $\Omega$ was first studied by Hung-Vuong \cite{H-V}. This result was later reproved in \cite{Ben-Gh}, where a precise relationship was established between the $m$-Lelong number $\nu_m(dd^c v, a)$ and the mean values of $v$ over the ball $\mathbb{B}_n(a,r)$ and the sphere $\mathbb{S}_n(a,r)$. More precisely,
$$
\nu_m(dd^c v,a)
= \lim_{r\to0}\frac{2\,\mathcal{M}(v,\mathbb{S}_n(a,r))}{\phi_{m,n}(r^2)}
= 2\left(1+\frac{1}{n}-\frac{1}{m}\right)
\lim_{r\to0}\frac{\mathcal{M}(v,\mathbb{B}_n(a,r))}{\phi_{m,n}(r^2)}.
$$

In particular, if $v$ is bounded near $a$, then $\nu_m(dd^c v,a)=0$. Since $v$ is $q$-subharmonic for every $1 \leq q \leq m$, it follows that
$$
\nu_q(dd^c v, a) = 0 \quad \text{for all } 1 \leq q < m.
$$

  To introduce the main topic of this paper, we begin by recalling the notion of slices. 
Let $\alpha_1$ (resp. $\alpha_2$) be a nonnegative function with compact support in $\mathbb{B}_p$ (resp. in $\mathbb{B}_{n-p}$) such that
$$
\int_{\mathbb{C}^p} \alpha_1 \, dV_p = \int_{\mathbb{C}^{n-p}} \alpha_2 \, dV_{n-p} = 1,
$$
where $dV_k$ denotes the Lebesgue measure on $\mathbb{C}^k$.\\

For $\varepsilon > 0$, we set
$$
\alpha_{1,\varepsilon}(z') = \frac{1}{\varepsilon^{2p}} \alpha_1\!\left(\frac{z'}{\varepsilon}\right),
\qquad
\alpha_{2,\varepsilon}(z'') = \frac{1}{\varepsilon^{2(n-p)}} \alpha_2\!\left(\frac{z''}{\varepsilon}\right).
$$

Let $R$ be a current of bidimension $(k,k)$ on $\Omega$, and let 
$\pi : \mathbb{C}^p \times \mathbb{C}^{n-p} \to \mathbb{C}^p$ be the canonical projection. 
Fix a point $a \in \Omega'$ and assume that $p < k \leq n$. 
Following \cite{BMS-El}, the slice of $R$ by $\pi$ at $a$ is defined, when it exists, as the weak limit in $\mathscr{D}'_{(k-p,k-p)}(\Omega)$ of
$$
\int_{\mathbb{B}_p(a,\varepsilon)\times \Omega''}
R \wedge \pi^*\!\left(
\frac{1}{\varepsilon^{2p}} \alpha_1\!\left(\frac{z'-a}{\varepsilon}\right) (dd^c |z'|^2)^p
\right) \wedge \varphi,
$$
as $\varepsilon \to 0$, for any test form $\varphi \in \mathscr{D}_{(k-p,k-p)}(\Omega)$. 
This slice is denoted by $\langle R,\pi,a\rangle_{\alpha_1}$.

The notion of slice $\langle R,\pi,a\rangle_{\alpha_1}$, originally introduced in \cite{BM-El}, was further developed in \cite{HK2, HK4, HK8} by replacing the standard quadratic weight $z' \mapsto |z'|^2$ with a locally bounded plurisubharmonic function depending only on $z' \in \mathbb{C}^p$. 
In the particular case where
$$
\varphi(z') = |z'|^2 \qquad \text{and} \qquad \alpha_1 = \frac{p!}{\pi^p}\mathbf{1}_{\mathbb{B}_p},
$$
this construction reduces to the classical definition of slices due to Federer \cite{Fe}. 
Similarly, when $\varphi(z') = |z'|^2$ and $\alpha_1 \in \mathscr{D}(\mathbb{C}^p)$, one recovers the notion of slices introduced by Harvey and Shiffman \cite{H-S}.

To avoid additional technical difficulties, we restrict ourselves to the framework of \cite{BM-El}, using the standard quadratic weight, and do not consider the more general slicing associated with plurisubharmonic weights as in \cite{HK2, HK4, HK8}.

\section{ Slicing of  $m$-subharmonic functions}
In this section, we study slicing properties of $m$-subharmonic functions, extending classical results from pluripotential theory to this more general setting. Slicing provides an effective tool for analyzing the behavior of such functions along lower-dimensional complex subspaces, in particular their restrictions and integrability properties.\\

The following result is a direct consequence of the local integrability of $m$-subharmonic functions. Although we are not aware of a specific reference in this context, its proof follows by a straightforward adaptation of the classical argument in \cite{BM-El}, and is therefore omitted.
\begin{proposition}\label{p3.1}
    Let $v\in \Sh_m(\Omega)$ be an $m$-subharmonic function. For $\varepsilon>0$, define   $\alpha_{\varepsilon}(z)=\alpha_{\varepsilon}(z',z'')=\alpha_{1,\varepsilon}(z')\alpha_{2,\varepsilon}(z'')$ and $v_{\varepsilon}:=v*\alpha_{\varepsilon}$. Then,
\begin{enumerate}
\item $\langle v_{\varepsilon},\pi,x'\rangle\;$ is well defined and  equals  $\;{v_{\varepsilon}}_{\mid\{x'\}\times\Omega''}$.
\item  $\lim_{\varepsilon\rightarrow0}\langle v_{\varepsilon},\pi,x'\rangle\;$ exists in $\;\mathscr D'(\Omega),\;$ if and only if $\;\langle v,\pi,x'\rangle\;$ exists.\\
In this case, we put $$\lim_{\varepsilon\rightarrow0}\langle v_{\varepsilon},\pi,x'\rangle=\langle v,\pi,x'\rangle=v_{\mid\{x'\}\times \Omega''}.$$
\end{enumerate}
    \end{proposition}
In the following example, we compare the index of subharmonicity of a function with that of its slices. 
\begin{example} Let $a,b\in\mathbb R$ and  $v_{a,b}$ be the function given by
    $$v_{a,b}(z)=a|z_1|^2+b|z_2|^2+\sum_{j=3}^n|z_j|^2.$$
    We know that the $m$-subharmonicity of $v_{a,b}$ on $\mathbb C^n$ is equivalent to the conditions: $S_k(v_{a,b})\geq 0$ for every $1\leq k\leq m$ where $$S_k(v_{a,b})=\binom{n-2}{k}+\binom{n-2}{k-1}(a+b)+\binom{n-2}{k-2}ab.$$ 
    Now, if $\pi:\mathbb{C}^n\longrightarrow\mathbb C$ is the canonical projection given by $\pi(z)=z_1$ (i.e. $p=1$) then the slice $u_b:=\langle v_{a,b},\pi,0\rangle$ is given by $$u_b(z_2,\dots,z_n)=v_{a,b}(0,z_2,\dots,z_n)=b|z_2|^2+\sum_{j=3}^n|z_j|^2.$$
    Thus $u_b$ is $k$-subharmonic on $\mathbb C^{n-1}$ if and only if $$b\geq -\frac{n-k-1}{k}.$$
    Figure \ref{figslice} corresponds to the case $n=5$. The general case is analogous; the only difference concerns the coefficients of the curves, while their geometric structure remains unchanged.
    
\begin{figure}[h]
\centering
\begin{minipage}{0.45\textwidth}
\centering
\includegraphics[width=0.7\linewidth]{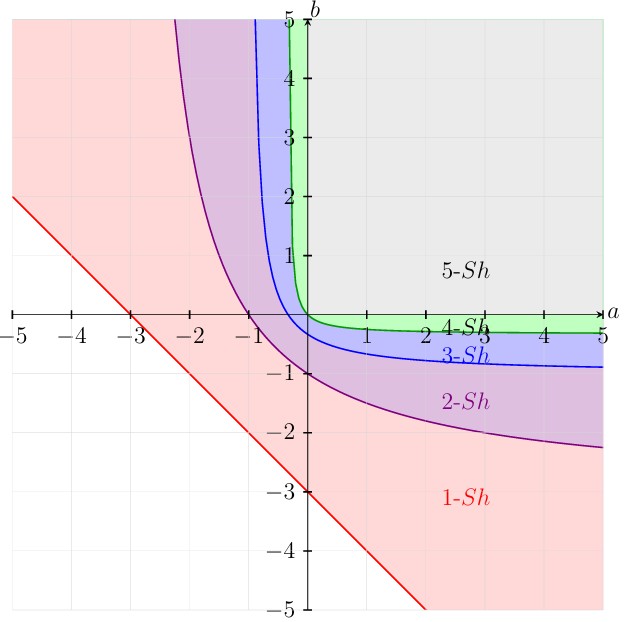}
\end{minipage}
\hfill
\begin{minipage}{0.45\textwidth}
\centering
\includegraphics[width=1\linewidth]{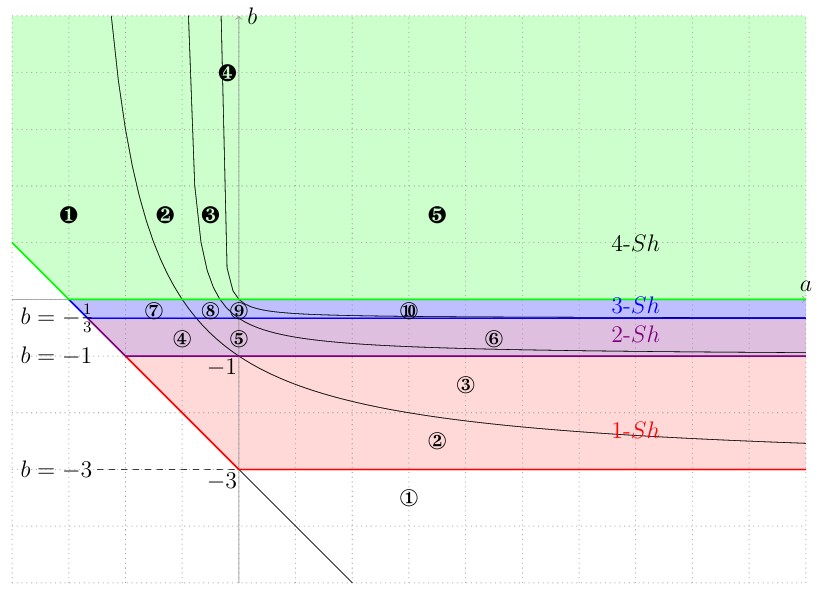}
\end{minipage}
\caption{$m$-subharmonicity of $v_{a,b}$ (left) and $k$-subharmonicity of $u_{b}$ (right) for $n=5$.}
\label{figslice}
\end{figure}

We claim that, if $m$ denotes the index of subharmonicity of $v_{a,b}$ and $k$ that of its slice $u_b$, then the set of parameters $(a,b)$ for which the function $v_{a,b}$ is subharmonic can be partitioned into $15$ distinct regions (See Figure \ref{figslice} at right). In each of these regions, the difference $k-(m-1)$ of the subharmonicity indices may be zero or positive. Table \ref{tab1} summarizes this classification. The motivation behind this example and the above classification is that the quantity $$k-(m-q_{m,p})$$ is always nonnegative. Establishing this property will be the aim of the main result.
\begin{table}[h]
    \centering
    \begin{tabular}{|l|ccccccccccccccc|}
    \hline
    Region& \ding{172}&\ding{173}&\ding{174}&\ding{175}&\ding{176}&\ding{177}&\ding{178}&\ding{179}& \ding{180}&\ding{181}&\ding{182}&\ding{183}&\ding{184}&\ding{185}&\ding{186}  \\
    \hline
    $m$& 1&1&2&1&2&3&1&2&3&4&1&2&3&4&5\\
    \hline
    $k$&not-Sh&1&1&2&2&2&3&3&3&3&4&4&4&4&4\\
    \hline
    $k-(m-1)$& & 1&0&2&1&0&3&2&1&0&4&3&2&1&0\\
    \hline
    \end{tabular}
    \caption{Comparison between the indices of subharmonicity of $v_{a,b}$ and its slice $u_b$.}
    \label{tab1}
\end{table}
\end{example}

We now establish a cutting theorem for $m$-subharmonic functions, showing that outside of a pluripolar exceptional set, the restriction is well defined and $(m-q_{m,p})$-subharmonic. This result plays a key role in the propagation of regularity and integrability properties across variables.
\begin{theorem}\label{t3.3}
Let $v \in \Sh_m(\Omega) $ be an $m$-subharmonic function and let $q_{m,p}$ be the smallest integer $q$ such that $q \geq \frac{mp}{n}$. We consider the set $$E=\left\{x'\in\cb^p:  \hbox{ the function } x''\longmapsto v(x',x'')\; \hbox{is not in}\;  L^1_{\loc}(\Omega'')\right\}.$$
 Then the following hold:
\begin{enumerate}
  \item For every point $\;x' \in \Omega' \smallsetminus E,\;$ the slice $\;\langle v, \pi, x' \rangle\;$ exists and  defines an $(m-q_{m,p})$-subharmonic function on $\Omega''$. We denote this restriction by $v_{|\{x'\} \times \Omega''}$
  \item For every $\;x' \in \Omega' \smallsetminus E,$ we have 
  $$
  \langle dd^c v, \pi, x' \rangle = dd^c\langle  v, \pi, x' \rangle =dd^c v_{|\{x'\} \times \Omega''}.
  $$
  \item The exceptional subset $\,E\,$ is pluripolar in $\;\mathbb{C}^p$.
\end{enumerate}
\end{theorem}
\begin{proof}
\textbf{Proof of statement (1).}\\ Fix the coordinates $\;(z_1,\dots z_n)=(z',z'')\;$ of $\cb^n\;$ near a point $\;x=(x',x'')\in(\Omega'\smallsetminus E)\times\Omega''$. Let $v_{\varepsilon}=v*\rho_{\varepsilon}$ be a regularization of $v$ by a smooth kernel  $(\rho_{\varepsilon})_{\varepsilon}$ so that
\begin{center} $v_{\varepsilon}\in\mathcal C^{\infty}(\Omega_{\varepsilon},\bR)\cap\Sh_m(\Omega_{\varepsilon})\quad$ where $\quad\Omega_{\varepsilon}=\{z\in\Omega:\mathrm{dist}(z,\Omega^c)>\varepsilon\}$.
\end{center}
Denote by  $\lambda_{1,\varepsilon}(z)\leq\dots\leq\lambda_{n,\varepsilon}(z)$  the eigenvalues of the complex Hessian matrix $$A_{\varepsilon}(z):=\left(\frac{\partial^2v_{\varepsilon}}{\partial z_j\partial\overline z_k}\right)_{1\leq j,k\leq n}$$  of $v_{\varepsilon}$ around $x$ and  let
$$\alpha_{\varepsilon}(z):=\alpha_{\varepsilon}=dd^cv_{\varepsilon}=\sum_{1\leq j,k\leq n}\frac{\partial^2v_{\varepsilon}}{\partial z_j\partial\overline z_k}idz_j\wedge d{\overline z}_k.$$
Since $v_{\varepsilon}$ is  $m$-subharmonic on  $\Omega_{\varepsilon}$,   we have
\begin{equation}\label{eq3.1}
\binom{n}{k}\alpha^k_\varepsilon\wedge\beta^{n-k}=S_k(A_{\varepsilon}(z))\beta^n,\quad\; k=1,\dots,m
\end{equation}
where $$S_k(A_{\varepsilon}(z))=\sum_{1\leq j_1<j_2<\dots<j_k\leq n}\lambda_{j_1,\varepsilon}\dots\lambda_{j_k,\varepsilon}\geq0\quad \forall\;k=1,\dots,m.$$
Which gives $$\lambda_\varepsilon:=(\lambda_{1,\varepsilon},\dots,\lambda_{n,\varepsilon})\in\Gamma_m$$ where $\Gamma_m$ is the G{\aa}ding cone.
In addition, since $v_{\varepsilon}$ is smooth, then an explicit computation of the slice  $\langle v_{\varepsilon},\pi,x'\rangle $, across $\cb^p$, at point $x'$,  shows that such a slice is well defined and is given by $$\varphi_{\varepsilon}(.):={v_{\varepsilon}}_{\mid\{x'\}\times\cb^{n-p}}$$  and satisfies
\begin{equation}\label{eq3.2}
dd^c\langle v_{\varepsilon},\pi,x'\rangle=\langle dd^cv_{\varepsilon},\pi,x'\rangle=dd^c{v_{\varepsilon}}_{\mid\{x'\}\times\cb^{n-p}}=dd^c\varphi_{\varepsilon}.
\end{equation}
Formula (\ref{eq3.2}) provides that
\begin{equation}\label{eq3.3}
\sum_{p+1\leq j,k\leq n}\frac{\partial^2\varphi_{\varepsilon}}{\partial z_j\partial\overline z_k}(z'')idz_j\wedge d\overline z_k=\sum_{p+1\leq j,k\leq n}\frac{\partial^2v_{\varepsilon}}{\partial z_j\partial\overline z_k}(x',z'')idz_j\wedge d\overline z_k
\end{equation}
and the  Hermitian matrices $B_{\varepsilon}(z'')$ and $A''_{\varepsilon}(x',z'')$ produced  by $\varphi_{\varepsilon}$ and ${v_{\varepsilon}}_{\mid\{x'\}\times\cb^{n-p}}$, respectively, are  given by $$B_{\varepsilon}(z''):=\left(\frac{\partial^2\varphi_{\varepsilon}}{\partial z_j\partial\overline z_k}(z'')\right)_{p+1\leq j,k\leq n}\quad\hbox{and}\quad A''_{\varepsilon}(x',z''):=\left(\frac{\partial^2v_{\varepsilon}}{\partial z_j\partial\overline z_k}(x',z'')\right)_{p+1\leq j,k\leq n}$$  and  are equal.\\
The G{\aa}ding inequality affirms that the projection of $\lambda_\varepsilon=(\lambda_{1,\varepsilon},\dots,\lambda_{n,\varepsilon})\in\Gamma_m$ to $\mathbb C^{n-p}$ 
satisfies $$\widetilde{\lambda}_\varepsilon=(\lambda_{j_1,\varepsilon},\dots,\lambda_{j_{n-p},\varepsilon})\in\Gamma_k,\quad \forall\;k\leq \frac{m(n-p)}{n}.$$ Therefore, we deduce that 
$$\binom{n-p}{s}{\alpha''_{\varepsilon}}^s\wedge{\beta''}^{n-p-s}=S_s(B_{\varepsilon}(z'')){\beta''}^{n-p}\geq0,\quad\forall\; 1\leq s\leq k$$
where 
$${\alpha''}_{\varepsilon}=\sum_{p+1\leq j,l\leq n}\frac{\partial^2v_{\varepsilon}}{\partial z_j\partial\overline z_l}idz_j\wedge d{\overline z}_l.$$
 It follows that, for every $q \geq q_{m,p}$, the restriction 
${v_\varepsilon}_{|\{x'\}\times\Omega''}$ is $(m-q)$-subharmonic on $\Omega''$. 
In particular, for $q = q_{m,p}$, the family 
$\big({v_\varepsilon}_{|\{x'\}\times\Omega''}\big)_\varepsilon$ is decreasing 
and consists of $(m-q_{m,p})$-subharmonic functions. Therefore, by the stability of $(m-q_{m,p})$-subharmonic functions under decreasing limits, 
there exists a unique $(m-q_{m,p})$-subharmonic function $\varphi$ on $\Omega''$ such that
$$\varphi = \lim_{\varepsilon \to 0} {v_\varepsilon}_{|\{x'\}\times\Omega''}.$$
By definition of slices, we may identify this limit with the restriction of $v$, namely
$$\langle v, \pi, x' \rangle = v_{|\{x'\}\times\Omega''}.$$
\textbf{Proof of statement (2).}\\
Since the family $\big({v_\varepsilon}_{|\{x'\}\times\Omega''}\big)_\varepsilon$ 
decreases to $v_{|\{x'\}\times\Omega''}$ and each 
${v_\varepsilon}_{|\{x'\}\times\Omega''}$ is $(m-q)$-subharmonic on $\Omega''$, 
it follows that $v_{|\{x'\}\times\Omega''}$ is also $(m-q)$-subharmonic on $\Omega''$. 

Therefore, by the continuity of the operator $dd^c$ under decreasing limits (see Formula~\eqref{eq3.2}), we obtain, in the sense of currents,
$$dd^c\langle v,\pi,x'\rangle
=\langle dd^c v,\pi,x'\rangle
= dd^c\big(v_{|\{x'\}\times\Omega''}\big).$$

\noindent \textbf{Proof of statement (3).}\\ Since pluripolarity is a local property, it suffices to prove that $E$ is locally pluripolar in $\Omega'$. Hence, we may assume that $E \subset \omega'$, where $\omega' \Subset \Omega'$ is a relatively compact open set. Let $\omega'' \Subset \Omega''$ be another relatively compact open set.\\
Since $E$ is a Borel subset of $\omega'$, we use the characterization of pluripolar sets in terms of pluricomplex capacity (see \cite{Be-Ta}). It is enough to show that
$$
Cap_p(E,\omega'):=\sup \left\{ \int_E (dd^c u)^p \;:\; u \in \mathrm{Psh}(\omega'),\; 0 \le u \le 1 \right\} = 0.
$$
Let $u \in \mathrm{Psh}(\omega')$ with $0 \leq u \leq 1$, and define $\tilde{u} := u \circ \pi$. Let $(v_j)_j$ be a decreasing sequence of smooth $m$-subharmonic functions converging pointwise to $v$. Without loss of generality, we assume that $v<0$ on $\omega=\omega'\times\omega''$. Then, by using the slicing formula of \cite{BM-El},
$$
\int_\omega v_j (dd^c \tilde{u})^p \wedge \beta''^{n-p} = \int_{x'\in\omega'} \langle v_j, \pi, x' \rangle (\mathbf{1}_{\omega''} \beta''^{n-p}) (dd^c u)^p = \int_{x'\in\omega'} f_j(x') (dd^c u)^p,
$$
where 
$$
f_j(x') := \int_{x''\in\omega''} v_j(x', x'') \beta''^{n-p}.
$$
By statement (1), the sequence $f_j \in \Sh_{\lfloor \frac{mp}{n}\rfloor}(\omega')$, decreases pointwise. Indeed, the function $f_j$ is an integral depending on a parameter $x'$ and the function $v_j(.,x'')$ is an $\lfloor \frac{mp}{n}\rfloor$-subharmonic for almost every $x''\in\Omega''$. Now, if $x' \in E$, then 
$$v_{|\{x'\} \times \omega''} \notin L^1_{\loc}(\omega'').$$
 So that  $f_j(x') \longrightarrow -\infty$ as $j\to+\infty$. Thus, for any $A > 0$, there exists $j_0$ such that for all $j \geq j_0$, $f_j(x') < -A$. 
It follows that 
\begin{align*}
    A\int_{E}(dd^cu)^p&\leq -\int_{E}f_j(x')(dd^cu)^p\\
    &\leq -\int_E \langle v_j, \pi, x' \rangle (\mathbf{1}_{\omega''} \beta''^{n-p}) (dd^c u)^p\\
    &\leq \int_\omega (-v_j) (dd^c \tilde{u})^p \wedge \beta''^{n-p}\\
    &\leq \int_\omega (-v) (dd^c \tilde{u})^p \wedge \beta''^{n-p}
\end{align*}
By taking a compact subset $K \Subset \Omega$ such that $\omega \Subset K$, the Chern--Levine--Nirenberg inequality yields a constant $c>0$, depending on $K$, such that 
$$\int_\omega (-v) (dd^c u)^p \wedge \beta''^{n-p}\leq c\|v\|_{L^1(K)}.$$
Thus, we deduce that 
$$\int_{E}(dd^cu)^p\leq \frac{c}{A}\|v\|_{L^1(K)}.$$
Letting $A\to+\infty$, we obtain 
$$\int_{E}(dd^cu)^p=0.$$
This shows that $Cap_p(E,\omega')=0$. i.e. $E$ is pluripolar in $\Omega'$.
\end{proof}

\section{ Slices and $m$-Lelong numbers of  $m$-subharmonic functions}
In this section, we investigate the relation between the slicing of $m$-subharmonic functions and the concept of $m$-Lelong numbers. Our method relies the construction of specially designed $m$-positive closed currents obtained by wedging the complex Hessian of the function with kernel potentials adapted to the slicing geometry. This construction yields sharp integral identities that describe how the local mass of $m$-subharmonic functions is distributed along complex slices, thus linking slicing techniques to the quantitative analysis of singularities. The constants appearing in these identities encode both the dimensional parameters $(n,m,p)$ and the radius of the slicing ball, reflecting the precise scaling behavior of the underlying kernels. 

\begin{example}
    For $\tau>1$ a positive real number, let $v_\tau$ be the function defined by $$v_\tau(z)=-\frac{1}{(\tau-1)|z|^{2(\tau-1)}}.$$ Then $v_\tau$ is $m$-subharmonic on $\mathbb C^n$ if and only if $\tau\leq \frac{n}{m}$. Indeed, it is easy to see that for every $1\leq k\leq m$,
    $$(dd^cv_\tau)^k\wedge\beta^{n-k}=\frac{1}{|z|^{2k\tau}}\left(\beta-\frac{k\tau}{|z|^2}d|z|^2\wedge d^c|z|^2\right)\wedge\beta^{n-1}.$$
    It is well known that the last quantity is positive for every $1\leq k\leq m$ if and only if  $m\tau\leq n$. 
    Moreover, using the result of \cite{Ben-Gh}, in such a case, we have 
    \begin{enumerate}
        \item $\nu_m(v_\tau,.)\equiv0$ on $\mathbb C^n$ for every $\tau<\frac{n}{m}$.
        \item $\nu_m(v_{\frac{n}{m}},.)\equiv0$ on $\mathbb C^n\smallsetminus\{0\}$ and $\nu_m(v_{\frac{n}{m}},0)>0$.
    \end{enumerate}
    This function is related to the fundamental solution. In fact, if $\tau=\frac{n}{m}$  then $v_{\frac{n}{m}}=\Phi_{n,m}$.\\
    By considering the restriction  
$$u_\tau := {v_\tau}_{|\{0\}\times\mathbb{C}^{n-p}}, \quad \text{i.e.} \quad 
u_\tau(z'') = v(0,z''), \quad \forall\, z'' \in \mathbb{C}^{n-p},$$
we observe that $u_\tau$ is $(m-q)$-subharmonic on $\mathbb{C}^{n-p}$ for every integer $q$ such that $$q \geq \frac{mp}{n}.$$
In particular, we have
$$\nu_{m-q}(u_\tau, \cdot) \equiv 0 \quad \text{on } \mathbb{C}^{n-p}$$
 for every $\tau \leq \frac{n-p}{m-q}$, unless $q = \frac{mp}{n}$ (whenever this quantity is an integer).
\end{example}
The $m$-Lelong number defined along complex $p$-planes exhibits strong rigidity properties. In particular, for fixed fibers, the function behaves uniformly outside a negligible exceptional set. The following result summarizes this framework, providing the key integral relation from which our subsequent results on
$m$-Lelong numbers follow.
\begin{proposition} \label{p4.2}
    Let $v$ be an $m$-subharmonic function on $\Omega'\times\Omega''$, $\mathcal B$ be a Borel relatively compact subset of $\Omega'$ and $x''\in\Omega''$. For every integer $q\geq \frac{mp}{n}$, Let $\nu_{m-q}(dd^cv,\mathcal B,x'',.)$ be the function defined by 
    $$\nu_{m-q}(dd^cv,\mathcal B,x'',r)=\frac{1}{r^{2(n-p)(1-\frac{1}{m-q})}}\int_{\mathcal B\times \mathbb B_{n-p}(x'',r)}dd^cv\wedge\beta^{n-1}.$$ Then $\nu_{m-q}(dd^cv,\mathcal B,x'',.)$ is a non-decreasing function on $]0,dist(x'',\partial \Omega'')[$.  The directional $(m-q)$-Lelong number of $dd^cv$ at $x''$ with respect to $\mathcal B$ is defined as $$\nu_{m-q}(dd^cv,\mathcal B,x'')=\lim_{r\to0}\nu_{m-q}(dd^cv,\mathcal B,x'',r).$$
\end{proposition}
\begin{proof} Without loss of generality, we can assume that $x''=0$. Let 
$$q\geq q_{m,p}:=\inf\left\{s\in\mathbb{N},\ s\geq \frac{mp}{n}\right\}=\left\{\begin{array}{lcl}
    \ds\frac{mp}{n} & \text{if}& \ds\frac{mp}{n}\in\mathbb N \\
    \left\lfloor\frac{mp}{n}\right\rfloor+1 & \text{if}&\ds \frac{mp}{n}\not\in\mathbb N 
\end{array}\right.$$ where $\lfloor.\rfloor$ is the integer part.\\
Using the binomial formula, we obtain
    \begin{align*}
        \nu_{m-q}(dd^cv,\mathcal B,0,r)&=\frac{1}{r^{2(n-p)(1-\frac{1}{m-q})}}\int_{\mathcal B\times \mathbb B_{n-p}(r)}dd^cv\wedge\beta^{n-1}\\
        &=\binom{n-1}{p} I(r)+\binom{n-1}{p-1} J(r)
    \end{align*} where
    $$I(r) =\frac{1}{r^{2(n-p)(1-\frac{1}{m-q})}}\int_{\mathcal B\times \mathbb B_{n-p}(r)}dd^cv\wedge\beta'^p\wedge\beta''^{n-p-1}$$ and 
    $$J(r)=\frac{1}{r^{2(n-p)(1-\frac{1}{m-q})}}\int_{\mathcal B\times \mathbb B_{n-p}(r)}dd^cv\wedge\beta'^{p-1}\wedge\beta''^{n-p}.$$
     Using the following $(m-q)$-subharmonic function on $\mathbb C^{n-p}$
$$\Phi(z''):=\Phi_{n-p,m-q}(z'')=-\frac{1}{\left(\frac{n-p}{m-q}-1\right)|z''|^{2\left(\frac{n-p}{m-q}-1\right)}},$$
and the  Stokes formula, we obtain
    \begin{align*}
        I(r)&=\frac{1}{r^{2(n-p)(1-\frac{1}{m-q})}}\int_{\mathcal B\times \mathbb B_{n-p}(r)}dd^cv\wedge\beta'^p\wedge\beta''^{n-p-1}\\
        &=\frac{1}{r^{2(n-p)(1-\frac{1}{m-q})}}\int_{\mathcal B\times \partial\mathbb B_{n-p}(r)}dd^cv\wedge\beta'^p\wedge d^c|z''|^2\wedge\beta''^{n-p-2}\\
        &=\int_{\mathcal B\times \partial\mathbb B_{n-p}(r)}dd^cv\wedge\beta'^p\wedge(dd^c\Phi)^{m-q-1}\wedge d^c|z''|^2\wedge\beta''^{n-p-m+q-1}.\\
    \end{align*}
    Hence for $0<r_1<r_2<\mathrm{dist}(0,\partial\Omega''),$ if we set $\mathbb B_{n-p}(r_1,r_2)=\mathbb B_{n-p}(r_2)\smallsetminus \mathbb B_{n-p}(r_1)$, we obtain a Lelong-Jensen type formula:
    \begin{equation}\label{eq4.1}
        \begin{array}{lcl}
      I(r_2)-I(r_1)&=&\ds \int_{\mathcal B\times \partial\mathbb B_{n-p}(r_2)}dd^cv\wedge\beta'^p\wedge(dd^c\Phi)^{m-q-1}\wedge d^c|z''|^2\wedge\beta''^{n-p-m+q-1}\\
       &&\ds-\int_{\mathcal B\times \partial\mathbb B_{n-p}(r_1)}dd^cv\wedge\beta'^p\wedge(dd^c\Phi)^{m-q-1}\wedge d^c|z''|^2\wedge\beta''^{n-p-m+q-1}\\
       &=&\ds \int_{\mathcal B\times \mathbb B_{n-p}(r_1,r_2)}dd^cv\wedge\beta'^p\wedge(dd^c\Phi)^{m-q-1}\wedge\beta''^{n-p-m+q}
    \end{array}
    \end{equation}
    With the same argument, 
    $$J(r)=\int_{\mathcal B\times \partial\mathbb B_{n-p}(r)}dd^cv\wedge\beta'^{p-1}\wedge(dd^c\Phi)^{m-q-1}\wedge d^c|z''|^2\wedge\beta''^{n-p-m+q}.$$
    Thus, 
    \begin{equation}\label{eq4.2}
        J(r_2)-J(r_1)=\ds \int_{\mathcal B\times \mathbb B_{n-p}(r_1,r_2)}dd^cv\wedge\beta'^{p-1}\wedge(dd^c\Phi)^{m-q-1}\wedge\beta''^{n-p-m+q+1}.
    \end{equation}
 The equalities \eqref{eq4.1} and \eqref{eq4.2} show that both functions $r\longmapsto I(r)$ and $r\longmapsto J(r)$ increase.
\end{proof}
\begin{corollary}\label{c4.3}
    Let $v$ be an $m$-subharmonic function on $\Omega'\times\Omega''$, $\mathcal B$ be a Borel relatively compact subset of $\Omega'$ and $x''\in\Omega''$. For every $q\geq q_{m,p}$, the directional $(m-q)$-Lelong number $\nu_{m-q}(dd^cv,\mathcal B,x'')$ and the $(m-q)$-Lelong number of the slice $\langle dd^cv,\pi,x'\rangle$ are related by the following identity:
    \begin{align*}
      \nu_{m-q}(dd^cv,\mathcal B,x'')&=\binom{n-1}{p}\int_{\mathcal B}\nu_{m-q}(\langle dd^cv,\pi,x'\rangle,(x',x''))\beta'^p(x')\\
      &=\binom{n-1}{p}\int_{\mathcal B}\nu_{m-q}(dd^c\langle v,\pi,x'\rangle,(x',x''))\beta'^p(x').  
    \end{align*}
\end{corollary} 
\begin{proof}
    Thanks to Proposition \ref{p4.2}, the directional Lelong number of $dd^cv$ is given by 
    $$\nu_{m-q}(dd^cv,\mathcal B,x'')=\binom{n-1}{p}\lim_{r\to0}I(r)+\binom{n-1}{p-1}\lim_{r\to0}J(r).$$
    Hence, to prove the result, it suffices to show that the second limit is equal to zero. To this aim, let $\varepsilon>0$ and $g_\varepsilon$ be the function defined on $\mathbb B_{n-p}(x'',1)$ by 
    $$g(z'')=\frac{1}{(-\log(|z''-x''|^2))^\varepsilon}.$$
    Then $g$ is a continuous plurisubharmonic 
    function on $\mathbb B_{n-p}(x'',1)$ and we have 
    $$dd^cg(z'')\geq \frac{\varepsilon(n-p-1)}{n-p}\frac{\beta''}{|z''-x''|^2(-\log(|z''-x''|^2))^{1+\varepsilon}}.$$
    Hence, for every $0<r<r_0<\min(1,dist(x'',\partial \Omega''))$, 
    \begin{align*}
        c&:=\int_{\mathcal B\times\mathbb B_{n-p}(x'',r_0)}dd^cv\wedge\beta'^{p-1}\wedge dd^cg\\
        &\geq\frac{\varepsilon(n-p-1)}{n-p}\int_{\mathcal B\times\mathbb B_{n-p}(x'',r)} \frac{dd^cv\wedge\beta'^{p-1}\wedge\beta''^{n-p}}{\left(|z''-x''|^2(-\log(|z''-x''|^2))^{1+\varepsilon}\right)^{n-p}}\\
        &\geq \frac{\varepsilon(n-p-1)}{n-p}\frac{1}{\left(r^2(-\log(r^2))^{1+\varepsilon}\right)^{n-p}}\int_{\mathcal B\times\mathbb B_{n-p}(x'',r)} dd^cv\wedge\beta'^{p-1}\wedge\beta''^{n-p}
    \end{align*}
Since $g$ is continuous, we deduce that the value $c$ of the above integral is finite and for every $0<r<r_0<\min(1,dist(x'',\partial \Omega''))$,
$$\int_{\mathcal B\times\mathbb B_{n-p}(x'',r)} dd^cv\wedge\beta'^{p-1}\wedge\beta''^{n-p}\leq  \frac{(n-p)c}{\varepsilon(n-p-1)}r^{2(n-p)}(-\log(r^2))^{(n-p)(1+\varepsilon)}.$$
    It follows that 
        \begin{align*}
            J(r)&=\frac{1}{r^{2(n-p)(1-\frac{1}{m-q})}}\int_{\mathcal B\times \mathbb B_{n-p}(x'',r)}dd^cv\wedge\beta'^{p-1}\wedge\beta''^{n-p}\\
            &\leq \frac{(n-p)c}{\varepsilon(n-p-1)}r^{\frac{2(n-p)}{m-q}}(-\log(r^2))^{(n-p)(1+\varepsilon)}.
        \end{align*} 
        This shows that $J(r)$ goes to zero when $r\to0$. As a consequence,  we obtain 
    \begin{align*}
        \lim_{r\to0}I(r)&=\int_{x'\in\mathcal B}\left(\lim_{r\to0}\frac{1}{r^{2(n-p)(1-\frac{1}{m-q})}}\int_{
    \{x'\}\times \mathbb B_{n-p}(x'',r)}\langle dd^cv,\pi,x'\rangle\wedge\beta''^{n-p-1}\right)\beta'^p\\
    &=\int_{\mathcal B}\nu_{m-q}(\langle dd^cv,\pi,x'\rangle,(x',x''))\beta'^p(x').
    \end{align*}
\end{proof}
We conclude this section by proving the following result that, up to a constant,   the Lelong number of the sliced current $\langle dd^cv,\pi,x'\rangle$ at a point coincides with the classical
$m$-Lelong number of $dd^cv.$ This equality, which holds almost everywhere, confirms the consistency of the slicing approach and reinforces the geometric meaning of the
$m$-Lelong number.
\begin{theorem}\label{t4.4}
    Let $v$ be an $m$-subharmonic function on $\Omega'\times\Omega''$ and $x''\in\Omega''$ be a fixed point. 
    \begin{enumerate}
        \item If $q_{m,p}=\frac{mp}{n}$ then $$\nu_m(dd^cv,(x',x''))=\pi^p\frac{\Gamma\left(n-\frac{n}{m}+1\right)}{\Gamma\left(n-\frac{n}{m}+p+1\right)}\binom{n-1}{p}\nu_{m-q_{m,p}}(\langle dd^cv,\pi,x'\rangle,(x',x''))$$ for almost every $x'\in\Omega'$.
        \item If $q_{m,p}>\frac{mp}{n}$, then $\nu_{m-q_{m,p}}(\langle dd^cv,\pi,x'\rangle,(x',x''))=0$ for almost every $x'\in\Omega'$.
    \end{enumerate}
\end{theorem}
\begin{proof}
    Let us start by proving the first assertion, where we assume that $\frac{mp}{n}$ is an integer. In this case, we have $q_{m,p}=\frac{mp}{n}$ and the slice $\langle v,\pi,x'\rangle$ is $(m-q_{m,p})$-subharmonic on $\Omega''$ for every $x'\in\Omega'\smallsetminus E$ ($E$ is an exceptional set). Hence, to prove the result, it suffices to show that 
    $$\int_{\mathcal B}\nu_m(dd^cv,(x',x''))\beta'^p(x')=\int_{\mathcal B}\nu_{m-q}(\langle dd^cv,\pi,x'\rangle,(x',x''))\beta'^p(x')$$
    for every Borel subset $\mathcal{B}$ of $\Omega'$. To this aim, one can assume that $x''=0$ and take $\mathcal{B}=\mathbb B_p(a,r_0)$.
    Using the definition of the Lelong number and the Fubini theorem, we see that for $r$ small enough,
    $$\begin{array}{l}
        \ds\int_{\mathbb B_p(a,r_0)}\nu_m(dd^cv,(x',0),r)\beta'^p(x')\\
        \ds=\frac{1}{r^{2n\left(1-\frac{1}{m}\right)}}\int_{\mathbb B_p(a,r_0)}\left(\int_{\mathbb B_n((x',0),r)}dd^cv\wedge\beta^{n-1}(z)\right)\beta'^p(x')\\
        =\ds\frac{1}{r^{2n\left(1-\frac{1}{m}\right)}}\int_{z\in A(r)\times \mathbb B_{n-p}(r)}\left(\int_{\mathbb B_p(z,\sqrt{r^2-|z''|^2})}\beta'^p(x')\right)dd^cv\wedge\beta^{n-1}(z)\\
        =\ds\frac{\pi^p}{p!r^{2n\left(1-\frac{1}{m}\right)}}\int_{A(r)\times \mathbb B_{n-p}(r)}(r^2-|z''|^2)^pdd^cv\wedge\beta^{n-1}(z)\\
    \end{array}$$
    where  $A(r)$ is a subset of $\mathbb C^p$ that satisfies $$\mathbb B_p(a,r_0-r)\times \mathbb B_{n-p}(r)\subset A(r)\times \mathbb B_{n-p}(r)\subset\mathbb B_p(a,r_0+r)\times \mathbb B_{n-p}(r).$$
    Now, using Proposition \ref{p4.2}, if we consider the positive measure $$\sigma(dd^cv,\mathbb B_p(a,r_1),r):=\int_{\mathbb B_p(a,r_1)\times \mathbb B_{n-p}(r)}dd^cv\wedge\beta^{n-1},$$
    we obtain for $0<r<\varepsilon<r_0$, 
    $$\begin{array}{l}
        \ds\int_{\mathbb B_p(a,r_0)}\nu_m(dd^cv,(x',0),r)\beta'^p(x')\\
        \leq\ds\frac{\pi^p}{p!r^{2n\left(1-\frac{1}{m}\right)}}\int_{\mathbb B_p(a,r_0+\varepsilon)\times \mathbb B_{n-p}(r)}(r^2-|z''|^2)^pdd^cv\wedge\beta^{n-1}(z)\\
        \leq \ds\frac{\pi^p}{p!r^{2n\left(1-\frac{1}{m}\right)}}\int_0^r(r^2-t^2)^pd\sigma(dd^cv,\mathbb B_p(a,r_0+\varepsilon),t)\\
        \leq \ds\frac{2\pi^p}{(p-1)!r^{2n\left(1-\frac{1}{m}\right)}}\int_0^r(r^2-t^2)^{p-1}t\sigma(dd^cv,\mathbb B_p(a,r_0+\varepsilon),t)dt\\
        \leq \ds\frac{2\pi^p}{(p-1)!r^{2n\left(1-\frac{1}{m}\right)}}\int_0^r(r^2-t^2)^{p-1}t^{1+2(n-p)\left(1-\frac{1}{m-q_{m,p}}\right)}\nu_{m-q_{m,p}}(dd^cv,\mathbb B_p(a,r_0+\varepsilon),t)dt\\
        \leq \ds\frac{2\pi^p}{(p-1)!r^{2n\left(1-\frac{1}{m}\right)}}\nu_{m-q_{m,p}}(dd^cv,\mathbb B_p(a,r_0+\varepsilon),r)\int_0^r(r^2-t^2)^{p-1}t^{1+2(n-p)\left(1-\frac{1}{m-q_{m,p}}\right)}dt\\
        \leq \ds \pi^p\frac{\Gamma\left(n-\frac{n-p}{m-q_{m,p}}+1\right)}{\Gamma\left(n-\frac{n-p}{m-q_{m,p}}+p+1\right)}\frac{r^{2\left(n-\frac{n-p}{m-q_{m,p}}\right)}}{r^{2n\left(1-\frac{1}{m}\right)}}\nu_{m-q_{m,p}}(dd^cv,\mathbb B_p(a,r_0+\varepsilon),r)\\
        \leq \ds c(n,m,p)\nu_{m-q_{m,p}}(dd^cv,\mathbb B_p(a,r_0+\varepsilon),r)
    \end{array}$$
    where $$c(n,m,p)=\pi^p\frac{\Gamma\left(n-\frac{n}{m}+1\right)}{\Gamma\left(n-\frac{n}{m}+p+1\right)}.$$
    By Corollary \ref{c4.3}, if we pass to the limit when $r\to0$, we obtain 
    \begin{align*}
        \int_{\mathbb B_p(a,r_0)}\nu_m(dd^cv,(x',0))\beta'^p(x')&\leq c(n,m,p)\nu_{m-q_{m,p}}(dd^cv,\mathbb B_p(a,r_0+\varepsilon),0)\\
        &\leq\ds c(n,m,p)\binom{n-1}{p}\int_{\mathbb B_p(a,r_0+\varepsilon)}\nu_{m-q_{m,p}}(\langle dd^cv,\pi,x'\rangle, (x',0))\beta'^p(x')
    \end{align*}
    for every $0<\varepsilon<r_0$ small enough. Thus, we conclude that 
    \begin{equation}\label{eq4.3}
        \int_{\mathbb B_p(a,r_0)}\nu_m(dd^cv,(x',0))\beta'^p(x')\leq \ds c(n,m,p)\binom{n-1}{p}\int_{\mathbb B_p(a,r_0)}\nu_{m-q_{m,p}}(\langle dd^cv,\pi,x'\rangle, (x',0))\beta'^p(x').
    \end{equation}
    On the other hand, using the same competitions, for every $0<r<\varepsilon$, one can find 
    $$\int_{\mathbb B_p(a,r_0)}\nu_m(dd^cv,(x',0),r)\beta'^p(x')\geq c(n,m,p)\binom{n-1}{p}\int_{\mathbb B_p(a,r_0-\varepsilon)}\nu_{m-q_{m,p}}(\langle dd^cv,\pi,x'\rangle, (x',0))\beta'^p(x').$$
    Again, this yields to 
    \begin{equation}\label{eq4.4}
        \int_{\mathbb B_p(a,r_0)}\nu_m(dd^cv,(x',0),r)\beta'^p(x')\geq c(n,m,p)\binom{n-1}{p}\int_{\mathbb B_p(a,r_0)}\nu_{m-q_{m,p}}(\langle dd^cv,\pi,x'\rangle, (x',0))\beta'^p(x').
    \end{equation}
    Inequalities \eqref{eq4.3} and \eqref{eq4.4} give the equality in  the first assertion.\\
    For the second assertion, we proceed as in the previous case, to obtain  
    $$\nu_{m-q_{m,p}}(dd^cv,\mathbb B_p(a,r_0-\varepsilon),r)\leq d(n,m,p)r^{2\left(\frac{n-p}{m-q_{m,p}}-\frac{n}{m}\right)} \int_{\mathbb B_p(a,r_0)}\nu_m(dd^cv,(x',0),r)\beta'^p(x')$$
    for every $0<r<\varepsilon<r_0$, where $$d(n,m,p)=\frac{\Gamma\left(n-\frac{n-p}{m-q_{m,p}}+p+1\right)}{\pi^p\Gamma\left(n-\frac{n-p}{m-q_{m,p}}+1\right)}.$$    
    Since $$\frac{n-p}{m-q_{m,p}}-\frac{n}{m}>0,$$
    we conclude that 
    $$\binom{n-1}{p}\int_{\mathbb B_p(a,r_0-\varepsilon)}\nu_{m-q_{m,p}}(\langle dd^cv,\pi,x'\rangle, (x',0))\beta'^p(x')=\nu_{m-q_{m,p}}(dd^cv,\mathbb B_p(a,r_0-\varepsilon),0)=0.$$
    Therefore, $\nu_{m-q_{m,p}}(\langle dd^cv,\pi,x'\rangle, (x',0))=0$ for almost every $x'\in\mathbb B_p(a,r_0-\varepsilon)$ and the proof is achieved by the arbitrariness of the ball $\mathbb B_p(a,r_0-\varepsilon)$ in $\Omega'$.
\end{proof}

{\bf Acknowledgments.} \textit{The authors thank the anonymous referee for their careful review and constructive suggestions that improved the paper. Special thanks are also due to Professor Hassine El Mir for helpful discussions.
} \\

{\bf Financial interests.} \textit{The authors declare they have no financial interests.}\\

{\bf Conflict of interest.} \textit{The authors declare they have no conflict of interest.}\\

{\bf Data availability.} \textit{No datasets were generated or analyzed during the current study.}

\end{document}